\documentclass[reqno,12pt,a4paper]{amsart}

\numberwithin{equation}{section}
\allowdisplaybreaks 

\usepackage[normalem]{ulem}

\usepackage{amssymb,eucal,mathrsfs,setspace,bm}
\usepackage[noadjust]{cite}

\usepackage{array}
\newcolumntype{C}{>{$}c<{$}} 

\usepackage[largesc,theoremfont]{newtxtext}      
\usepackage[libertine,cmbraces,varbb]{newtxmath} 
\begingroup
  \makeatletter
  \@for\theoremstyle:=definition,remark,plain\do{%
    \expandafter\g@addto@macro\csname th@\theoremstyle\endcsname{%
      \addtolength\thm@preskip{.5\baselineskip plus .2\baselineskip minus .2\baselineskip}
      \addtolength\thm@postskip{.5\baselineskip plus .2\baselineskip minus .2\baselineskip}
    }%
  }
\endgroup

\usepackage{enumitem}
\setitemize{leftmargin=*}   
\setenumerate{leftmargin=*, 
  label=\textup{(\alph*)},  
	ref=\textup{\alph*}}      

\usepackage{mathtools} 
\usepackage{graphicx}

\usepackage[colorlinks=true,citecolor=red,linkcolor=blue]{hyperref} 

\usepackage[capitalise,noabbrev]{cleveref} 

\renewcommand{\cong}{\simeq}

\DeclareMathOperator{\id}{id}

\DeclareMathOperator{\tr}{tr} 

\DeclarePairedDelimiter{\brac}{\lparen}{\rparen}   
\DeclarePairedDelimiter{\sqbrac}{\lbrack}{\rbrack} 

\DeclarePairedDelimiterX{\comm}[2]{\lbrack}{\rbrack}{#1 , #2}  
\DeclarePairedDelimiterX{\acomm}[2]{\lbrace}{\rbrace}{#1 , #2} 
\DeclarePairedDelimiterX{\inner}[2]{\langle}{\rangle}{#1 , #2} 
\DeclarePairedDelimiterX{\super}[2]{\lparen}{\rparen}{#1 \delimsize\vert \mathopen{} #2} 

\newcommand{\fld}[1]{\mathbb{#1}}    
\newcommand{\alg}[1]{\mathfrak{#1}}  
\newcommand{\grp}[1]{\mathsf{#1}}    
\newcommand{\aalg}[1]{\mathsf{#1}}   
\newcommand{\Mod}[1]{\mathcal{#1}}   
\newcommand{\VOA}[1]{\mathsf{#1}}    

\newcommand{\ZZ}{\fld{Z}}

\newcommand{\CC}{\fld{C}}

\newcommand{\ag}{\alg{g}}
\newcommand{\ah}{\alg{h}}
\newcommand{\al}{\alg{l}}
\newcommand{\ap}{\alg{p}}

\newcommand{\ab}{\alg{b}}

\newcommand{\rlat}{\grp{Q}}                                        

\newcommand{\aai}{\aalg{I}}
\newcommand{\aau}{\aalg{U}}

\newcommand{\aaw}{\aalg{W}}
\newcommand{\aac}{\aalg{C}}
\newcommand{\aat}{\aalg{T}}

\newcommand{\envalg}[1]{\aau\brac{#1}}             

\newcommand{\uaffvoa}[2]{\VOA{V}_{#1} \brac{#2}}                
\newcommand{\affvoa}[2]{\VOA{L}_{#1} \brac{#2}}                 

\newcommand{\zhu}[1]{\mathsf{Zhu}\sqbrac*{#1}}                              

\newcommand{\mdc}{\Mod{C}}
\newcommand{\mdh}{\Mod{H}}
\newcommand{\mdl}{\Mod{L}}
\newcommand{\mdm}{\Mod{M}}
\newcommand{\mdn}{\Mod{N}}

\newcommand{\hw}{highest weight } 

\theoremstyle{plain}
\newtheorem{theorem}{Theorem}
\newtheorem{corollary}[theorem]{Corollary}
\newtheorem{lemma}[theorem]{Lemma}
\newtheorem{proposition}[theorem]{Proposition}

\newtheorem*{definition}{Definition}
\newtheorem*{remark}{Remark}

\usepackage[all]{xy}
\DeclareMathOperator{\Supp}{supp}

\newcommand{\supp}[1]{\Supp(#1)} 
\newcommand{\lam}{\lambda}
\newcommand{\Lam}{\Lambda}

\newcommand{\ch}{\mathrm{ch}}
\newcommand{\ann}[2]{\mathrm{Ann}_{#2}(#1)} 
\newcommand{\ug}{\envalg{\alg{g}}} 

\newcommand{\fg}{\bar{\ag}}
\newcommand{\fh}{\bar{\ah}}
\newcommand{\fb}{\bar{\ab}}
\newcommand{\fp}{\bar{\ap}}
\newcommand{\fl}{\bar{\al}}
\newcommand{\ufg}{\envalg{\fg}}
\newcommand{\fDelta}{\bar{\Delta}}
\newcommand{\fO}{\bar{\mathcal{O}}}

\newcommand{\fLam}{\bar{\Lambda}}

\newcommand{\frho}{\bar{\rho}}
\newcommand{\fmdv}{\bar{\mdv}}
\newcommand{\fmdl}{\bar{\mdl}}
\newcommand{\fmdh}{\bar{\mdh}}
\newcommand{\fmdm}{\bar{\mdm}}
\newcommand{\mdr}{\Mod{R}}

\newcommand{\fmdc}{\bar{\mdc}}
\newcommand{\wlat}{\grp{P}}                 
\newcommand{\fwlat}{\bar{\wlat}}
\newcommand{\frlat}{\bar{\rlat}}
\newcommand{\hbrst}{H^{\mathrm{BRST}}}

\newcommand{\mdv}{\Mod{V}}
\newcommand{\Znn}{\mathbb{Z}_{\geq0}}
\newcommand{\cspan}{\mathrm{span}_{\CC}}
\newcommand{\Cl}{Cl}
\newcommand{\comp}{\aalg{C}}
\newcommand{\cF}{\mathcal{F}}
\newcommand{\wuniva}{W^k(\fg,f_{\alpha_\ell})}
\newcommand{\wsimpa}{W_k(\fg,f_{\alpha_\ell})}
\newcommand{\wfina}{W^{\mathrm{fin}}(\fg,f_{\alpha_\ell})}
\newcommand{\wunivamod}{\wuniva\mbox{-mod}}
\newcommand{\cO}{\mathcal O}
\newcommand{\hlie}{H^{\mathrm{Lie}}}
\newcommand{\semi}{\mathrm{s.s.}}
\newcommand{\adm}{\mathrm{Adm}}
\newcommand{\sing}{\mathrm{Sing}}
\newcommand{\singtwo}{\mathrm{sing}}

\makeatletter
\renewcommand\author@andify{%
  \nxandlist {\unskip ,\penalty-1 \space\ignorespaces}%
    {\unskip {} \@@and~}%
    {\unskip \penalty-2 \space \@@and~}%
}
\makeatother

\begin{document}

\title[]{Relaxed highest-weight  modules III: Character formulae}

\author[K~Kawasetsu]{Kazuya Kawasetsu}
\address[Kazuya Kawasetsu]{
Priority Organization for Innovation and Excellence\\
Kumamoto University\\
Kumamoto, Japan, 860-8555.
}
\email{kawasetsu@kumamoto-u.ac.jp}

\begin{abstract}
This is the third of a series of articles devoted to the study of relaxed
\hw  modules over vertex operator algebras.
Relaxed \hw  modules over affine vertex algebras associated to higher rank Lie algebras 
$A_\ell$ are extensively studied.
In particular, the string functions of simple relaxed \hw  modules
  whose top spaces are   simple cuspidal $A_\ell$-modules are shown to be the quotients by a power of the Dedekind eta series of the $q$-characters of simple ordinary modules over affine W-algebras associated with the minimal nilpotent elements of $A_\ell$.
\end{abstract}

\maketitle

\onehalfspacing

\section{Introduction} \label{sec:intro}


{\em Relaxed \hw  modules} are generalisation of \hw  modules and  
are believed to be building blocks of many logarithmic conformal field theory (CFT).
In particular, they are essential in the study of so-called {\em standard Verlinde formulas} for logarithmic CFT
launched in \cite{CR1}.

The relaxed \hw  modules over ``rank 1'' vertex algebras have been extensively studied
in both mathematics and physics litertures,
e.g.~\cite{Ada19,AM,CKLR,CreMod13,FST,G,KawRel18,MO,RW2}. 
(Recall that vertex algebras are chiral algebras of CFTs.)
Moreover, many fusion rules for bosonic ghost CFT predicted by the standard Verlinde formula 
and Nahm-Gaberdiel-Kausch algorithm in \cite{RW} were proved in \cite{AP}. 

Now investigation of relaxed \hw  modules over higher rank vertex algebras
is desired towards establishing  theory of higher rank logarithmic CFT.

Let $\uaffvoa{k}{\fg}$ be the universal {\em affine vertex algebra} of non-critical level $k$
associated to the finite-dimensional simple Lie algebra $\fg$ of rank $\ell$ with the Cartan subalgebra $\fh$ of $\fg$.
Recall that modules over $\uaffvoa{k}{\fg}$ are smooth modules of level $k$ over the affine Kac-Moody algebra
$\ag$ associated to $\fg$.
Let $\ah$ be the Cartan subalgebra of $\ag$ associated to $\fh$
and $\mdm$ a {\em weight} $\uaffvoa{k}{\fg}$-module, a module which is semisimple over $\ah$ 
with finite-dimensional weight spaces.
A {\em relaxed \hw  vector} is a weight vector $v\in\mdm$ such that
$(\fg[t]t).v=0$.
If $\mdm$ is generated by a relaxed \hw  vector, $\mdm$ is called 
a relaxed \hw  module \cite{FST,RW2}.

Let $\affvoa{k}{\fg}$ be the simple quotient of $\uaffvoa{k}{\fg}$.
In the second of this series \cite{KawRel19}, classification of simple relaxed \hw  modules
over $\affvoa{k}{\fg}$ of arbitrary rank is established subject to that of
simple \hw  modules.
Note that for admissible levels, 
simple \hw  modules over $\affvoa{k}{\fg}$ are classified in \cite{A16}. 
(See also \cite{AFL} for a study of positive energy modules including relaxed \hw  modules when $\fg=A_2$ and $k$ is admissible.)
However, the characters of such modules are not known except for a few cases
(see \cite{Ada16} and \cite{KRW} for $\fg=A_2$ and $k=-3/2$).

In this paper, we study characters of relaxed \hw  modules over affine vertex algebras 
associated to $\fg=A_\ell$ for all $\ell\geq 1$.

Let us explain our results briefly.
We say that a relaxed \hw  module over $\uaffvoa{k}{\fg}$ 
 has a {\em cuspidal top space}
if its top space is a {\em cuspidal module}
over $\fg$.
A cuspidal module over $\fg$ is by definition a weight $\fg$-module not parabolically induced from a module
over some Levi subalgebra $\fl\subsetneq \fg$.
The simple cuspidal modules are completely classified in \cite{MatCla00} and are fundamental in the classification of all weight modules over $\fg$ \cite{FerLie90,MatCla00}.
Note that by the result of \cite{FerLie90}, simple cuspidal modules exist (if and) only if $\fg=A_\ell$ or $C_\ell$.
For simplicity, we only consider relaxed \hw  modules with cuspidal top spaces in this paper.

Let $\fg$ be a finite-dimensional simple Lie algebra of type $A_\ell$
with fixed Cartan and Borel subalgebras $\fh\subset \fb\subset \fg$.
Let $\frlat$ and $\fwlat^+$ be the root lattice and set of dominant integral weights of $\fg$
with simple roots $\alpha_1,\ldots\alpha_\ell$ and fundamental weights 
$\omega_1\ldots,\omega_\ell$, as usual.
We denote by $\fmdl_\lam$ the simple \hw  $\fg$-module of highest weight $\lam\in\fh^*$.
Identifying $\fh^*\cong\fh$, the semisimple element $x=\omega_\ell$ of $\fg$ induces the grading
$$
\fg=\fg_{-1}\oplus \fg_0\oplus \fg_1,\qquad \fg_n:=\{v\in\fg\,|\,[x,v]=nv\}
$$
on $\fg$. 

Let $\cO_k$ be the BGG category of $\uaffvoa{k}{\fg}$ and 
$\cO^0_k$ the parabolic BGG category of $\uaffvoa{k}{\fg}$, the full subcategory
of $\cO_k$ of all objects integrable 
with respect to $\fg_0$.
The simple objects of $\cO^0_k$ are the simple \hw  $\uaffvoa{k}{\fg}$-modules
$\mdl_\Lam$ of highest weights $\Lam\in \wlat^{0,+}_k$, where
$\wlat^{0,+}_k$ is the set of all elements of $\ah^*$ of level $k$ which are integrable with respect to
$\fg_0$.
Let $\fLam$ denote the usual projection of $\Lam\in\ah^*$ onto $\fh^*$.

The main results of this paper lie in certain connections between simple relaxed \hw  modules 
with cuspidal top spaces and ordinary modules over the universal {\em minimal W-algebra} $\wuniva$.

Recall {\em Mathieu's twisted localisation functor} $\Phi_{\mu+\frlat}(\cdot)$.
It sends all infinite-dimensional \hw  modules over $\fg$, integrable with respect to $\fg_0$,
 to cuspidal $\fg$-modules of weight support $\mu+\frlat$
 and finite-dimensional modules to zero.
We shall apply the functor to modules in category $\cO^0_k$ over $\uaffvoa{k}{\fg}$. 
Let $\lam$ be an element of $\fh^*$ which is not dominant integral but integrable with respect to $\fg_0$.
Recall the set $\sing(\lam)$ of all $\mu+\frlat\in\fh^*/\frlat$ 
such that  $\Phi_{\mu+\frlat}(\fmdl_\lam)$ is not simple.
$\sing(\lam)$ is completely determined in \cite{MatCla00} and 
  it is a union of $\ell+1$ distinct codimension one cosets of $\fh^*/\frlat$.  

Let us also recall the  BRST ``$-$'' reduction functor $\hbrst_0(\cdot)$ from the category $\cO^0_k$ 
to the category of ordinary modules over $\wuniva$.
By \cite{Aii}, the collection of $\hbrst_0(\mdl_\Lam)$ with 
 $\Lam\in\wlat^{0,+}_k$ such that $\fLam\not\in\fwlat^+$ is the complete list
 of simple ordinary modules over $\wuniva$.
 If $\Lam\in\wlat^{0,+}_k$ with $\fLam\in\fwlat^+$, then $\hbrst_0(\mdl_\Lam)=0$.
The main theorem of the present paper is the following.

\begin{theorem}\label{sec:thmsummary1}
\begin{enumerate}
\item The list of all $\Phi_{\mu+\frlat}(\mdl_\Lam)$ with (1) $\Lam\in\wlat^{0,+}_k$ such that $\fLam\not\in\fwlat^+$
and (2) $\mu+\frlat\in(\fh^*/\frlat)\setminus\sing(\fLam)$, is the complete list of simple relaxed \hw  modules with cuspidal top spaces
over $\uaffvoa{k}{\fg}$.
If $\Lam\in\wlat^{0,+}_k$ with $\fLam\in\fwlat^+$, then $\Phi_{\mu+\frlat}(\mdl_\Lam)=0$
for all $\mu\in\fh^*$.


\item For any $\Lam\in\wlat^{0,+}_k$ and $\mu\in\fh^*$, we have the formula
$$
\ch[\Phi_{\mu+\frlat}(\mdl_\Lam)](y,q,z)=y^k\frac{\ch[\hbrst_0(\mdl_\Lam)](q)}{\eta(q)^{2\ell}}
\sum_{\beta\in\frlat}z^{\mu+\beta}
$$
of the full character of $\Phi_{\mu+\frlat}(\mdl_\Lam)$ using the $q$-character of $\hbrst_0(\mdl_\Lam)$.
\end{enumerate}
\end{theorem}

%
%

Note that by the cohomology vanishing,   $\ch[\hbrst_0(\mdl_\Lam)](q)$
is expressed in terms of Kazhdan-Lusztig polynomials \cite{Aii}. 
We also note that the second statement of \cref{sec:thmsummary1}
is non-trivial in the sense that the BRST ``$-$'' reduction of $\Phi_{\mu+\frlat}(\mdl_\Lam)$ 
is zero if the rank $\ell$ of $\fg$ is greater than 1
and $\mu+\frlat\in(\fh^*/\frlat)\setminus\sing(\fLam)$.  
 
On the other hand, the appearance of the minimal W-algebra is natural since the associated variety of the annihilator of a simple cuspidal representation is the minimal nilpotent orbit closure \cite{MatCla00}.
 
When $\fg=A_1$ and $\Lam$ is an admissible weight, the character formula above reduces to the Creutzig-Ridout character formula  \cite{CreMod13}  (see also \cite{KawRel18,Sat}).

To prove the theorem, recall that the characters of simple modules in $\cO^0_k$ are expressed 
in terms of {\em parabolic Verma modules} (and the so-called parabolic Kazhdan-Lusztig polynomials).
We apply the exact functors to the formula and derive character formulas
of relaxed \hw  modules with cuspidal top spaces over $\uaffvoa{k}{\fg}$ (\cref{sec:coherentKL})
and simple ordinary modules over $\wuniva$ \eqref{eqn:charw1}.

When $\fg=A_2$ and $k$ is an admissible level with denominator 2, the simple W-algebra
$\wsimpa$ is rational \cite{A13} and called a Bershadsky-Polyakov vertex algebra.
(Besides $\fg=A_1$, they are the only rational minimal W-algebras of admissible levels 
  \cite{A15,KW08}.  )
In this case, we show a similar 
correspondence of simple relaxed \hw  modules with cuspidal top spaces
over $\affvoa{k}{\fg}$ and simple ordinary modules over the simple W-algebra $\wsimpa$
(see \cref{sec:thmsummary}). 
As a result, we have a {\em modular invariance property} of string functions of simple relaxed \hw 
modules with cuspidal top spaces over $\affvoa{k}{\fg}$ as follows.
For a relaxed highest weight module $\mdr$ with cuspidal top spaces, let us denote  by $s[\mdr](q)$
the unique non-zero {\em string function} of $\mdr$ (see \cref{sec:char1} for the definition
of string functions).

\begin{theorem}\label{sec:modularinv}
Let $\fg$ be of type $A_2$ and $k$ an admissible level with denominator 2.
Then for any simple relaxed \hw  module $\mdr$ with cuspidal top spaces over $\affvoa{k}{\fg}$,
the series $s[\mdr](q)$ converges to a holomorphic function on the complex upper-half plane
$\mathbb H$ with $q=e^{2\pi \sqrt{-1}\tau}$, $\tau\in\mathbb H$.
Moreover, the vector space spanned by the holomorphic functions
$$
\eta(q)^{4}s[\mdr](q)\qquad (q=e^{2\pi\sqrt{-1}\tau},\tau\in\mathbb H)
$$
for all simple relaxed \hw  modules $\mdr$ with cuspidal top spaces over $\affvoa{k}{\fg}$
is finite-dimensional and invariant under the slash action $|_0$ of $SL_2(\ZZ)$:
\[
(f|_0\gamma)(\tau)=f\Bigl(
\frac{a\tau+b}{c\tau+d}\Bigr),\qquad \gamma=
\begin{pmatrix}a&b\\c&d\end{pmatrix}\in SL_2(\ZZ).
\]
\end{theorem}

A similar modular invariance of string functions of simple relaxed \hw  modules
with cuspidal top spaces over $\affvoa{k}{A_1}$ of admissible levels is crucial for
{\em Creutzig-Ridout
modular invariance} of full characters in \cite{CreMod13}, yet it is not fully formulated mathematically rigorously.
In \cite{KRW}, the modular invariance of full characters for $\affvoa{-3/2}{A_2}$ is examined and
fusion rules predicted from the standard Verlinde formula are observed to be non-negative.
We hope to come back to this point in the future works as well as relaxed \hw 
modules whose top space is parabolically induced from cuspidal modules.

%
%

The paper is organised as follows.
In \cref{sec:va1}, we recall definitions of vertex algebras and their modules.
Then, in \cref{sec:setting}, we define affine vertex algebras and in \cref{sec:kl1}, we recall 
 a character formula of simple \hw modules in $\cO^0_k$ in terms of parabolic Verma modules, 
which is crucial for our proof. 
 The theory of cuspidal modules over $A_\ell$ is recalled in \cref{sec:cuspidal}.
\cref{sec:relaxed} is devoted to relaxed \hw  modules with cuspidal top spaces
and a character formula of such modules in terms of relaxed Verma modules.
In particular, in \cref{sec:tw}, we show (a) in \cref{sec:thmsummary1}.
In \cref{sec:minimal}, we recall representation theory of the minimal W-algebra
$\wuniva$ and show a character formula for ordinary modules by using one in \cref{sec:kl1}.
The assertions (b) and (c) in \cref{sec:thmsummary1} are shown in \cref{sec:w} and \cref{sec:charord}.
Finally, in \cref{sec:main}, we prove \cref{sec:modularinv}.
Some technical assertions for twisted localisation are shown in \cref{sec:appendix}.

\subsection*{Acknowledgment}
The author would like to thank Tomoyuki Arakawa, Shu Kato, Olivier Mathieu, 
Arun Ram and David Ridout for helpful discussions and advice.
This work is partially supported by 
MEXT Japan ``Leading Initiative for Excellent Young Researchers (LEADER)'',
JSPS Kakenhi Grant numbers 19KK0065 and 19J01093 and 
Australian Research Council Discovery Project DP160101520.
\section{Vertex algebras}\label{sec:va1}

\subsection{Vertex algebras}\label{sec:va}

Let $V$ be a vector space.
We set
\[
V[[z,z^{-1}]]=\Bigl\{\sum_{n\in\ZZ}v_n z^n\,|\,v_n\in V\Bigr\},\qquad
V[[z]]=\Bigl\{\sum_{n=0}^\infty v_n z^n\,|\,v_n\in V\Bigr\},
\]
where $z$ is an indeterminate.
A power series $a(z)\in \mathrm{End}_{\CC}(V)[[z,z^{-1}]]$ 
 is called a {\em field}
if for any $b\in V$, there is $N\in\ZZ$ such that $a(z)b\in V[[z]]z^N$.

A {\em vertex algebra} is a vector space 
$V$ 
equipped with
a vector $\bm{1}\in V$, called the {\em vaccum vector},
a linear map $\partial:V\rightarrow V$ (the {\em translation operator})
and
fields $Y(a,z)\in \mathrm{End}_{\CC}(V)[[z,z^{-1}]]$ ($a\in V$),
satisfying the following axioms:
\begin{itemize}
\item $Y(\bm{1},z)=\id_V$, \qquad $Y(a,z)\bm{1}\in a+V[[z]]z$ \quad ($a\in V$),\qquad $\partial\bm{1}=0$,
\item $[\partial,Y(a,z)]=\partial_z Y(a,z)$, 
\item (locality) for any $a,b\in V$, there is $N>0$ such that 
$(z-w)^NY(a,z)Y(b,w)=0$.
%
\end{itemize}

Let $V$ be a vertex algebra.
For $a\in V$, we write $Y(a,z)=\sum_{n\in \ZZ}a(n)z^{-n-1}$ with $a(n)\in \mathrm{End}(V)$.
A {\em hamiltonian} of a vertex algebra $V$ is
a diagonalizable linear operator $H:V\rightarrow V$ such that
\[
[H,Y(a,z)]=Y(Ha,z)+z\partial_z Y(a,z), \qquad a\in V.
\]




Let $V$ be a vertex algebra with hamiltonian $H$ and suppose that $H$-eigenvalues are non-negative integers. 
Let $\omega$ be an element of $V$.
The pair $(V,\omega)$ is called a {\em vertex operator algebra} (VOA)
of central charge $c\in\CC$ if
\begin{itemize}
\item For any $n,m\in\ZZ$,
$$
[L_n,L_m]=(n-m)L_{n+m}+\frac{n^3-n}{12}\delta_{n,-m}c.
$$
Here, we write $Y(\omega,z)=\sum_{n\in\ZZ}L_nz^{-n-2}$.

\item
$L_{-1}=\partial$ and $L_0=H$.

\item
$V_0=\CC|0\rangle$ and $\omega\in V_2$.

\end{itemize}
The element $\omega$ is called the {\em conformal vector} of a VOA $(V,\omega)$.




\subsection{Modules over vertex algebras}

Let $V$ be a vertex algebra with hamiltonian $H$.
A {\em module} over $V$ is a vector space $M$ equipped with fields $Y_M(a,z)\in\mathrm{End}(M)[[z,z^{-1}]]$
 for all $a\in V$
such that
\begin{itemize}
\item
$Y^M(\bm{1},z)=\id_M$,
\item
$Y^M(a,z)\in \mathrm{End}(M)[[z,z^{-1}]]$, $a\in V$,
\item (Borcherds identity)
For any $a,b\in V$ and $m,n,r\in \ZZ$, we have
\begin{align*}
& \sum_{i=0}^\infty \begin{pmatrix}m\\ i\end{pmatrix}
(a(r+i)b)^M(m+n-i)\\
&\quad=\sum_{i=0}^\infty (-1)^i\begin{pmatrix} r\\ i\end{pmatrix}
\Bigl( a^M(m+r-i)b^M(n+i)-(-1)^rb^M(n+r-i)a^M(m+i)\Bigr).
\end{align*}
Here, for $a\in V$, the operators $a^M(n)$ are defined by
$$
Y^M(a,z)=\sum_{n\in\frac{1}{2}\ZZ}a^M(n)z^{-n-1}.
$$
\end{itemize}

A module $M$ is called {\em graded} if $M$ is equipped with a diagonalisable operator $H^M$ on $M$,
called a hamiltonian, which satisfies
$$[H^M,Y^M(a,z)]=Y^M(Ha,z)+z\partial_zY^M(a,z)\qquad a\in V.$$
A graded module $M$ decomposes into the sum $M=\bigoplus_{d\in\CC}M_d$
with $M_d=\{v\in M\,|\,H^Mv=dv\}$ for all $d\in\CC$.
The $H^M$-eigenvalues are called conformal weights.

If $V$ is a VOA and $M$ is a graded module, then all hamiltonians on $M$ are 
$\{L^M_0+d\,|\,d\in\CC\}$. 
Here, $L^M_0$ is the zero mode of $\omega$: $Y^M(\omega,z)=\sum_{n\in\ZZ}L_n^Mz^{-n-2}$.

A graded module $M$ is called {\em positive energy}
if there is a finite subset $S\subset \CC$ such that $H^M$-eigenvalues
belong to $S+\ZZ_{\geq0}$.
Let $M$ be a positive energy module.
Then we have a finite $S\subset \CC$ such that
(1) $M=\bigoplus_{h\in S,n\geq 0}M_{h+n}$, (2) $M_h\neq 0$ for all $h\in S$ and
(3) there are no integral differences between elements of $S$.
The {\em top space} $M_{top}$ of $M$ is then  defined by
$$
M_{top}=\bigoplus_{h\in S}M_h.
$$
Note that if $M$ is simple, then $S$ is a singleton.

A positive energy module $M$ is called {\em ordinary}
if $\dim M_d<\infty$ for any $d\in\CC$.

Let $V$ be a VOA 
of central charge $c$.
Let $M$ be an ordinary module with hamiltonian $H^M$ on $M$.
The {\em $q$-character} of $M$ is the series
$$
\ch[M](q)=\sum_{n\in\CC}(\dim M_n) q^{n-c/24}.
$$

\subsection{Zhu algebras}

Let $V$ be a vertex operator algebra.
Let $\envalg{V}$ denote the universal enveloping algebra
of $V$ as in \cite{MNT}.
We write $\envalg{V}_0$ and $\envalg{V}_<$ the subalgebras of $\envalg{V}$ which consist of
elements which preserve conformal weights and which lower conformal weights, respectively.
The {\em Zhu algebra} of $V$ is the associative algebra
$$
\zhu{V}=\envalg{V}_0/\envalg{V}_0\cap \envalg{V}\envalg{V}_<.
$$
Let $M$ be a positive energy $V$-module.
Then the top space $M_{top}$ is a $\zhu{V}$-module, which is denoted by $\zhu{M}$.

\begin{lemma}[\cite{ZhuMod96}] \label{sec:zhu}
The assignment $M\mapsto \zhu{M}$ defines a one-to-one correspondence
between the simple positive energy $V$-modules
and simple $\zhu{V}$-modules.
\end{lemma}

%
%
%

\section{Setting}\label{sec:setting}

\subsection{Graded Lie algebras}\label{sec:graded}
Let $\fg=A_\ell$ be a finite-dimensional simple Lie algebra of type $A_\ell$
with fixed Cartan and Borel subalgebras $\fh\subset \fb\subset \fg$.
The simple roots of the root system $\fDelta$ of $\fh^*$ are $\alpha_1,\alpha_2,\ldots,\alpha_\ell$
and the corresponding fundamental weights are $\omega_1,\omega_2,\ldots,\omega_\ell$.
The set of positive roots is $\fDelta^+$.
We denote the root lattice, weight lattice and the set of dominant integral weights of $\fh^*$ 
by $\frlat$, $\fwlat$ and $\fwlat^+$, respectively.
We fix root vectors $e_\beta,f_\beta$ of roots $\beta,-\beta$ for all $\beta\in\fDelta$.
The semisimple element $x=\omega_\ell$ of $\fg$ induces the grading
\begin{equation}\label{eqn:goodgrading}
\fg=\fg_{-1}\oplus \fg_0\oplus \fg_1,\qquad \fg_n:=\{v\in\fg\,|\,[x,v]=nv\}
\end{equation}
on $\fg$. 
Setting $\fDelta_n:=\{\beta\in\fDelta\,|\,[\alpha_n,\beta]=n\}$, we have
$$
\fg_{-1}=\cspan \{f_\beta\,|\,\beta\in \fDelta_1\},\quad \fg_0=A_{\ell-1}\oplus \CC \alpha_\ell,
\quad \fg_1=\cspan\{e_\beta\,|\,\beta\in \fDelta_1\},
$$
where $A_{\ell-1}\subset A_\ell$ is generated by $e_{\alpha_i},\alpha_i,f_{\alpha_i}$, $1\leq i\leq \ell-1$.
Note that 
\begin{equation}\label{eqn:delta2}
\fDelta_1=\{\beta_1,\beta_2,\ldots,\beta_\ell\},\qquad
\beta_i=\sum_{j=i}^\ell\alpha_j\quad (1\leq i\leq \ell).
\end{equation}
The Lie subalgebra $\fg_{\geq 0}:=\bigoplus_{n\geq 0}\fg_n$ is a parabolic subalgebra
of $\fg$ with the associated Levi subalgebra $\fg_0$.

Let $\fO^0$ be the parabolic BGG category with respect to the grading
\eqref{eqn:goodgrading}:
the full subcategory of the BGG category $\fO$ whose objects are
 direct sums of finite-dimensional $\fg_0$-modules. 
The simple modules in category $\fO^0$ are precisely the
irreducible \hw  modules $\fmdl_\lam$ of highest weight  $\lam\in \fwlat^{0,+}$, where
\begin{align*}
\fwlat^{0,+}:=\{\lam\in \fh^*\,|\,\langle \lam,\beta^\vee\rangle\in\Znn,
\mbox{ for all }\beta\in\fDelta_0\cap\fDelta^+\}.
\end{align*}
Since $\fDelta_0=\{\alpha_1,\ldots,\alpha_{\ell-1}\}$, 
we have
\begin{equation}\label{eqn:p0+}
\fwlat^{0,+}=\bigoplus_{i=1}^{\ell-1}\Znn \omega_i\oplus \CC \omega_\ell.
\end{equation}

Let $\lam$ be an element of $\fwlat^{0,+}$.
Let $\fmdh_\lam$ denote the irreducible finite-dimensional module over $\fg_0$
of highest weight $\lam$.
Letting $\fg_1$ acts trivially on $\fmdh_\lam$, we define the {\em parabolic Verma module}
of highest weight $\lam$ to be
$$
\fmdv^0_{\lam}=\envalg{\fg}\otimes_{\envalg{\fg_{\geq0}}}\fmdh_{\lam}.
$$
The $\fmdv^0_\lam$ is an object of $\fO^0$.

\subsection{Affine Kac-Moody algebras}

Let $\ag$ be the untwisted 
affine Kac-Moody algebra of $\fg$:
\[
\ag=\fg\otimes \CC[t,t^{-1}]\oplus \CC K\oplus \CC D
\]
with the Cartan subalgebra $\ah=\fh\oplus \CC K\oplus \CC D$
and the Borel subalgebra
$\ab=\fb[t]+\fg [t]t+\CC K+ \CC D$.
The root system of $\ag$ is $\Delta$ and the set of positive roots is $\Delta^+$.

We write by $\mdv_\Lam$ and $\mdl_\Lam$ 
the Verma module and irreducible \hw  module over $\ag$ of \hw  $\Lam\in\ah^*$,
respectively.

As usual, we have simple roots $\alpha_0,\alpha_1,\ldots,\alpha_\ell$ of $\ag$
and fundamental weights are $\Lambda_0,\Lambda_1,\ldots,\Lambda_\ell$.
The vector $\delta\in \ah^*$ is dual to $D$.
The space $\fh^*$ is embedded in $\ah^*$ and
we have the decomposition $\ah^*=\fh^*\oplus \CC \Lam_0\oplus \CC \delta$
and the projection $\ah^*\rightarrow \fh^*$, $\Lam\mapsto \fLam$.

A {\em weight module} over $\ag$ is a module on which $\ah$ acts semisimply and such that 
every weight spaces are finite-dimensional.
We write $\mdm(\Lam)=\{v\in\mdm\,|\,h(v)=\Lam(h)v,h\in\ah\}$ for $\Lam\in\ah^*$
and a weight module $\mdm$.

Let $\Lam$ be an element of $\ah^*$. If $\Lam(K)=k$, then $\Lam$ is called of level $k$.
We write the space of elements of $\ah^*$ of level $k$ by $\ah^*_k$.
A module $\mdm$ is called of level $k$ if $K$ acts as the scalar multiple $k$ on $\mdm$.

Let $\cO_k$ denote the BGG category of $\ag$ at level $k$.
The simple \hw  module $\mdl_\Lam$ and Verma module $\mdv_\Lam$
of \hw  $\Lam\in\ah^*_k$ are modules in category $\cO_k$.

Let $\cO^0_k$ denote the full subcategory of $\mathcal{O}_k$ 
 of all
objects which are direct sums of finite-dimensional $\fg_0$-modules.
We set
\[
\wlat^{0,+}_k=\{\Lam\in\ah^*_k\,|\, \fLam\in\fwlat^{0,+}\}.
\]
%

Let $\Lam$ be an element of $\wlat^{0,+}_k$.
Then $\Lam$ has the form $\Lam=\lam+d\delta+k\Lam_0$ with $d\in\CC$ and $\lam=\fLam$.
The {\em parabolic Verma module} of highest weight $\Lam$ is the induced module
$$
\mdv^0_{\Lam}=\envalg{\ag}\otimes_{\envalg{\fg_{\geq0}\oplus\fg[t]t\oplus\CC K\oplus \CC D}}{\fmdh_{\lam}},
$$ 
where $\fg_1\oplus\fg[t]t=0$, $D=d$ and $K=k$ on $\fmdh_{\lam}$.
The simple quotient of $\mdv^0_\Lam$ is isomorphic to
$\mdl_\Lam$ and we have
$\mdv^0_\Lam=\envalg{\ag}\otimes_{U(\fg[t]+\CC K+\CC D)}\fmdv^0_{\lam}$.

%
%

\subsection{Affine vertex algebras}

Let $k$ be a non-critical complex number and $\CC_k$ the one-dimensional representation of the Lie algebra 
$\fg[t]\oplus \CC K\oplus \CC D$
defined by $\fg[t]=0$, $K=k$ and $D=0$ on $\CC_k$.
The {\em universal affine vertex algebra} is a $\Znn$-graded vertex algebra on the generalised Verma module
$$
\uaffvoa{k}{\fg}=\envalg{\ag}\otimes_{\envalg{\fg[t]\oplus\CC K\oplus\CC D}}\CC_k
$$
with hamiltonian $H=-D$.
The simple quotient vertex algebra is denoted by $\affvoa{k}{\fg}$ and called the {\em simple affine vertex algebra}.


The $\uaffvoa{k}{\fg}$-modules with hamiltonian are in one-to-one correspondence with the {\em smooth} $\ag$-modules of level $k$ on which $D$ acts semisimply such that
the hamiltonian is $-D$.
Here, a $\ag$-module $\mdm$ is called smooth if for any $v\in \mdm$ and $a\in\fg$,
there is $N\in\ZZ$ such that if $n\geq N$ then $at^n.v=0$.
In particular, modules in category $\cO_k$
are $\uaffvoa{k}{\fg}$-modules. 


The Zhu algebra of $\uaffvoa{k}{\fg}$ is isomorphic to $\envalg{\fg}$
while that of $\affvoa{k}{\fg}$ is isomorphic to $\envalg{\fg}/\aai_k$
with a two-sided ideal $\aai_k$ of $\envalg{\fg}$.
Let $\mdm$ be a simple positive energy $\uaffvoa{k}{\fg}$-module.
Then $\mdm$ is a $\affvoa{k}{\fg}$-module if and only if $\aai_k$ acts as zero
on $\zhu{\mdm}$.

The affine vertex algebras are VOAs 
of central charge
\begin{equation}\label{eqn:caff}
c=\frac{k\dim\fg}{k+h^\vee}.
\end{equation}
Here, $h^\vee=\ell+1$ is the dual Coxeter number of $\fg$.


\section{Characters of simple modules in $\cO^0_k$}\label{sec:kl1}

In this Section, we recall that the characters
of simple modules in the parabolic BGG category $\cO^0_k$ of $\uaffvoa{k}{\fg}$
are expressed in terms of parabolic Verma modules
describing a linkage principle of the character formula.

\subsection{Character formulas in terms of Verma modules}

Let $\aaw$ denote the Weyl group of $\ag$.
The {\em Weyl vector} $\rho\in \ah^*$ is the sum of 
fundamental weights of $\ag$.
The {\em dot action} of $\aaw$ is 
\[
w\circ \Lam=w(\Lam+\rho)-\rho,\qquad w\in \aaw, \Lam \in\ah^*. 
\]

Let $k\neq -h^\vee$ be a non-critical complex number.
Let $[\mdm]$ denote the image of $\mdm\in\mathcal{O}_k$ in the Grothendieck group (characters) $[\mathcal{O}_k]$ of $\mathcal{O}_k$.

Let $\Lam$ be an element of $\ah^*_k$.
Let $\Delta(\Lam)$ be the integral root system and $\Delta^+(\Lam)=\Delta(\Lam)\cap\Delta^+$. 
The integral Weyl group for $\Lam$ is denoted by $\aaw(\Lam)$. 
The Bruhat order on $\aaw(\Lam)$ is written by $\geq_\Lam$.
Let $\aac(\Lam)$ be the subgroup of $\aaw(\Lam)$ generated by
$\{s_\alpha\,|\,\alpha\in\Delta^+,(\alpha^\vee,\Lam+\rho)=0\}$.

Let $\ah^*_{k,+}$ and $\ah^*_{k,-}$ denote the set of dominant and anti-dominant weights of
level $k$:
\begin{align*}
& \ah^*_{k,\pm}:=\{\Lam\in\ah^*_k\,|\,\pm(\alpha^\vee,\Lam+\rho)\in\ZZ_{\geq0}
\mbox{ for any }
\alpha\in\Delta^+(\Lam)\},
\end{align*}
where double sign corresponds.

Now the result of \cite{KasCha00} tells the following.
Let $\Lam$ be an element of $\ah^*_k$.
Then the set $\aaw(\Lam)\circ\Lam$ contains an element 
$\Omega$ of $\ah^*_{k,+}\cup \ah^*_{k,-}$.
Moreover, for any $w\in\aaw(\Omega)$ there is a unique $x\in w\aac(\Omega)$ such that
its length with respect to $\aaw(\Omega)$ is the largest (resp., smallest) among the 
elements of $w\aac(\Omega)$. 
The $x$  is called the longest (resp., shortest) element of $w\aac(\Omega)$.
If $\Omega\in\ah^*_{k,+}$, then
for any $w\in \aaw(\Omega)$ which is the longest element of $w\aac(\Omega)$, we have
\begin{equation}\label{eqn:kl1}
[\mdl_{w\circ\Omega}]=\sum_{\aaw(\Omega)\ni y\geq_\Omega w} a_y^\Omega [\mdv_{y\circ \Omega}],
\end{equation}
with integers $a_y^\Omega$ ($y\in\aaw(\Omega)$).
Similarly, if $\Omega\in\ah^*_{k,-}$, then for any $w\in \aaw(\Omega)$ which is the shortest
element of $w\aac(\Lam)$, we have
$$
[\mdl_{w\circ\Omega}]=\sum_{\aaw(\Omega)\ni y\leq_\Omega w} a_y^\Omega [\mdv_{y\circ \Omega}].
$$
The numbers $a_y^\Omega$ satisfy $a_w^\Omega=1$ and we can express them using the Kazhdan-Lusztig polynomials.

%

%
 
\subsection{Character formulas in terms of parabolic Verma modules}\label{sec:parakl}

Let $\Lam$ be an element of $\wlat^{0,+}_k$.
It is known that $[\mdl_\Lam]$ has the form
\begin{equation}\label{eqn:parabolicKL}
[\mdl_\Lam]=\sum_{\Omega\in \wlat^{0,+}_k} c_{\Lam,\Omega}
[\mdv^0_{\Omega}],\qquad c_{\Lam,\Lam}=1,\qquad c_{\Lam,\Omega}\in\ZZ.
\end{equation}

Later on, we shall apply exact functors to \eqref{eqn:parabolicKL} to have 
character formulas in \cref{sec:coherentKL} and
\eqref{eqn:charw1}.
In the rest of this section, we describe a linkage principle of the formula so that we have well-defined
sums in the right-hand sides of \eqref{eqn:parabolicKL}, \cref{sec:coherentKL} and
\eqref{eqn:charw1}.

\begin{remark}
The number $c_{\Lam,\Omega}$ is known to be expressed in terms of
{\em parabolic Kazhdan-Lusztig polynomials}.
See e.g.~\cite{Sor,Fie}.
\end{remark}

We have an element $\Omega\in\ah^*_{k,+}\cup\ah^*_{k,-}$ such that $\Lam\in\aaw(\Omega)\circ\Omega$.
Suppose $\Omega\in\ah^*_{k,+}$ and let $w$ be an element of $\aaw(\Omega)$ 
which is the longest among   $w\aac(\Omega)$ and satisfies $w\circ\Omega=\Lam$.
By inverting the Kazhdan-Lusztig formula, we have
$$
[\mdv_{w\circ\Omega}]=\sum_{\aaw(\Omega)\ni y\geq_{\Omega}w}b_y[\mdl_{y\circ\Omega}],
$$
with non-negative integers $b_y$ such that $b_w=1$.
Since $\mdv^0_{w\circ\Omega}$ is integrable with respect to $\fg_0$ and is a quotient of $\mdv_{w\circ\Omega}$,
it follows
$$
[\mdv^0_{w\circ\Omega}]=\sum_{\aaw(\Omega)\ni y\geq_{\Omega}w,\ y\circ\Omega\in\wlat^{0,+}_k}c_y[\mdl_{y\circ\Omega}],
$$
with non-negative integers $c_y$ such that $c_w=1$.
By inverting the formula, we have a linkage principle
\begin{equation}\label{eqn:link}
[\mdl_{w\circ\Omega}]=\sum_{\aaw(\Omega)\ni y\geq_{\Omega}w,\ y\circ\Omega\in\wlat^{0,+}_k}d_y[\mdv^0_{y\circ\Omega}],
\end{equation}
with $d_y\in\ZZ$ such that $d_w=1$.

Recall that the Weyl group $\aaw$ is a semi-direct product of the finite Weyl group
$\bar\aaw$ of $\fg$ and the group $\aat$ of translations $t_\alpha:\ah^*\rightarrow\ah^*$
($\alpha\in \frlat^\vee$), defined by
$$
t_\alpha(\Upsilon)=\Upsilon+\Upsilon(K)\alpha-\Bigl((\Upsilon|\alpha)+\frac{1}{2}|\alpha|^2 \Upsilon(K)\Bigr)\delta\qquad \Upsilon\in\ah^*.
$$
Here, $\frlat^\vee$ is the coroot lattice of $\fg$.
Let $y$ be an element of $\aaw$. Then $y$ has the form $y=rt_\alpha$ with $r\in\bar\aaw$ and $\alpha\in \frlat^\vee$.
Since $\Omega$ is dominant, we see that the minimal conformal weight of $\mdv^0_{y\circ\Omega}$,
which is the number $(y\circ\Omega)(-D)$,
tends to increase when $|\alpha|$ increases.
Therefore, the sum in the right-hand side of \eqref{eqn:link} is well-defined.

Suppose $\Omega\in\ah^*_{k,-}$ and let $w$ be an element of $\aaw(\Omega)$ 
which is the shortest among $w\aac(\Omega)$ and satisfies $w\circ\Omega=\Lam$.
Then in a similar way, we have a linkage principle
$$
[\mdl_{w\circ\Omega}]=\sum_{\aaw(\Omega)\ni y\leq_{\Omega}w,\ y\circ\Omega\in\wlat^{0,+}_k}d_y[\mdv^0_{y\circ\Omega}],
$$
which is a finite sum \cite[Lemma 2.3]{KasCha00}.

\section{Cuspidal modules over $A_\ell$}\label{sec:cuspidal}

\subsection{The cuspidal modules}

In this Subsection, we suppose that $\fg$ is an arbitrary simple Lie algebra.
A $\fg$-module $\fmdm$ is called a {\em weight module} if 
$$
\fmdm=\bigoplus_{\lam\in\fh^*}\fmdm(\lam),\qquad \fmdm(\lam):=\{v\in\fmdm\,|\,hv=\lam(h)v,h\in\fh\},
$$
and $\dim\fmdm(\lam)<\infty$ for each $\lam\in\fh^*$.

Let $\fmdm$ be a weight module. 
The {\em weight support} of $\fmdm$ is $\supp{\fmdm}=\{\lam\in\fh^*\,|\,\fmdm(\lam)\neq0\}$.
The dimension of $\fmdm(\lam)$ is called the {\em weight multiplicity} of $\lam\in\fh^*$.
We call $\fmdm$ {\em bounded} if $\fmdm$ is infinite-dimensional and
weight multiplicities are bounded above. 
Suppose that $\fmdm$ is bounded.
The maximum multiplicity is called the {\em degree} of $\fmdm$
and denoted by $\deg\fmdm$.

A weight module $\fmdm$ is called {\em dense} if the support of $\fmdm$ has the form
$\supp{\fmdm}=\lam+\frlat$ with $\lam\in\fh^*$. 
The dense modules are {\em cuspidal}, in the sense that they are weight modules not parabolically
induced from modules over Levi subalgebras $\al\subsetneq\fg$.

The following theorem by Fernando shows that simple cuspidal modules
are simple dense modules.

\begin{theorem}[\cite{FerLie90}]\label{fernando}
(1) Any simple weight modules over $\ag$ is isomorphic to the simple quotient of
the module $\envalg{\fg}\otimes_{\envalg{\fp}}(\fmdm)$
for some parabolic subalgebra $\fp$ of $\fg$ and simple dense weight module 
$\fmdm$ over the Levi subalgebra $\fl$ associated to $\fp$.
(2) Simple cuspidal modules exist only when $\fg=A_\ell$ or $C_\ell$.
\end{theorem}

\subsection{Coherent families}

The simple cuspidal modules are completely classified in \cite{MatCla00}
using {\em coherent families} of $\fg$.

\begin{definition}[\cite{MatCla00}]
A weight module $\fmdc$ over $\fg$ is a {\em coherent family} of $\fg$
if

\noindent
(1) $\supp{\fmdc}=\fh^*$,

\noindent
(2) $\dim \fmdc(\lam)=d$ for some $d>0$ and all $\lam\in\fh^*$,

\noindent
(3) For any $v\in \envalg{\fg}$, the trace $\tr_{\fmdc(\mu)}v$ is polynomial in $\mu\in\fh^*$.
\end{definition}

Let $\fmdc$ be a coherent family and set $\fmdc_{\lam+\frlat}=\bigoplus_{\mu\in\lam+\frlat}
\fmdc(\mu)$ for $\lam+\frlat\in\fh^*/\frlat$.
The coherent family $\fmdc$ is called {\em irreducible} if there is $\mu\in\fh^*$ such that 
$\fmdc_{\mu+\frlat}$ is an irreducible $\fg$-module.
The $\fmdc$ is called {\em semisimple} if it is a direct sum of simple $\fg$-modules.

Let $\fmdc$ be an irreducible semisimple coherent family of $\fg$.
Let $\singtwo(\fmdc)$ be the set of all $\mu+\frlat\in\fh^*/\frlat$ 
such that  $\fmdc_{\mu+\frlat}$ is not a simple cuspidal module.
The $\singtwo(\fmdc)$ is completely determined in \cite{MatCla00}
and 
  it is a union of $\ell+1$ distinct codimension one cosets of $\fh^*/\frlat$.

\begin{proposition}[\cite{MatCla00}]\label{sec:coherentcuspidal}
The list of $\fmdc_{\mu+\frlat}$ with all irreducible semisimple coherent family
$\fmdc$ of $\fg$ and $\mu+\frlat\in(\fh^*/\frlat)\setminus\singtwo(\fmdc)$ is the complete list
of simple cuspidal modules.
\end{proposition}


\subsection{Mathieu's twisted localisation}\label{sec:mathieu}

From now on, recall the settings in \Cref{sec:setting}.
Let $S$ be the multiplicative subset of $\envalg{\fg}$ generated by 
$f_{\beta_1},\ldots,f_{\beta_\ell}$ (recall \eqref{eqn:delta2}).
Since $S$ is commutative, we have (Ore's) localisation $S^{-1}\envalg{\fg}$ of $\envalg{\fg}$ as in \cite{MatCla00}.

Let $\lam$ be an element of $\fwlat^{0,+}$.
We see that $\fmdv^0_\lam$ is a bounded module of degree 
\begin{equation}\label{eqn:deg}
\deg\fmdv^0_\lam=\dim \fmdh_\lam
\end{equation}
since $\fmdv^0_\lam$ is freely generated by $f_{\beta_1},\ldots,f_{\beta_\ell}$ from $\fmdh_\lam$.
Therefore, if $\lam\not\in\fwlat^+$, then $\fmdl_\lam$ is bounded while if
$\lam\in\fwlat^+$, it is finite-dimensional.
The elements of $S$ act injectively on $\fmdv^0_\lam$ and if $\lam\not\in\fwlat^+$, so do on $\fmdl_\lam$.
In these cases, the localisation
$S^{-1}\fmdl_\lam$ and $S^{-1}\fmdv^0_\lam$ are dense modules of weight support $\lam+\frlat$ and
they have the uniform
weight multiplicity 
$$\dim (S^{-1}\fmdl_\lam)(\mu)=\deg\fmdl_\lam \ \mbox{ and } 
\dim (S^{-1}\fmdv^0_\lam)(\mu)=\deg\fmdv^0_\lam\quad \mbox{ for all }\mu\in\lam+\frlat.
$$
If $\lam\in\fwlat^+$, then $S^{-1}\fmdl_\lam=0$.


Set $\fmdm=\fmdl_\lam,\fmdv^0_\lam$.
Let $x=(x_1,\ldots,x_\ell)$ be an element of $\CC^\ell$.
As in \cite{MatCla00}, we define the twist of $S^{-1}\fmdm$ by $f^x:=f_{\beta_1}^{x_1}f_{\beta_2}^{x_2}\cdots f_{\beta_\ell}^{x_\ell}$ using exponentials. 
Then the twist $f^x.S^{-1}\fmdm$ is again a dense module with
$\supp{f^x.S^{-1}\fmdm}=-\mu_x+\lam+\frlat$, where
$\mu_x=x_1\beta_1+\cdots +x_\ell\beta_\ell$.
We write 
$$
\Phi(\fmdm)=\bigoplus_{x=(x_1,\ldots,x_\ell)\in\CC^\ell, 0\leq Re(x_1),\ldots,Re(x_\ell)<1}
f^x.S^{-1}\fmdm.
$$
The functor $\Phi$ is called {\em (Mathieu's) twisted localisation functor}.
If $S^{-1}\fmdm$ is non-zero, then $\Phi(\fmdm)$ is a coherent family of $\fg$ with the same degree as $\fmdm$.
The semisimplification of $\Phi(\fmdm)$ is denoted by
$\Phi^{\semi}(\fmdm)$.

Let $\mu+\frlat$ be an element of $\fh^*/ \frlat$.
We have a $\fg$-submodule 
$$
\Phi_{\mu+\frlat}(\fmdm):=\bigoplus_{\nu\in\mu+\frlat}\Phi(\fmdm)(\nu)
$$
of $\Phi(\fmdm)$. 
Here, $\Phi(\fmdm)(\nu)$ is the weight space of weight $\nu$ of $\Phi(\fmdm)$.
The functor $\Phi_{\mu+\frlat}$ is also called a twisted localisation functor.

\subsection{Classification of simple cuspidal modules using $\fwlat^{0,+}$}

Let us take over the notations from the last Subsection.

\begin{theorem}\label{thm:coherent}
The modules $\Phi^\semi(\fmdl_\lam)$ with $\lam\in\fwlat^{0,+}\setminus\fwlat^+$ are the complete list of irreducible semisimple coherent families of $\fg$.
If $\lam\in\fwlat^+$, then $\Phi^\semi(\fmdl_\lam)$ is zero.
\end{theorem}

\begin{proof}
By classification of simple cuspidal modules \cite[Theorem 8.6]{MatCla00} and \cite[Lemma 8.3]{MatCla00}, we see that the modules $\Phi^\semi(\fmdl_\lam)$ with $\lam\in\fwlat^{0,+}\setminus\fwlat^+$
exhaust all irreducible semisimple coherent families.

Let $\lam,\mu$ be elements of  $\fwlat^{0,+}\setminus\fwlat^+$ such that $\Phi^\semi(\fmdl_\lam)\cong \Phi^\semi(\fmdl_\mu)$.
It remains to show $\lam=\mu$.
By \cite[Theorem 8.6]{MatCla00} again, 
we see that $\mu=s_\ell\circ\lam$.
Here, $s_\ell$ is the simple reflection associated to $\alpha_\ell$.
The $\lam$ has the form $\lam=k_1\omega_1+\cdots+k_\ell\omega_\ell$
with $k_1,\ldots,k_{\ell-1}\in \Znn$ and $k_\ell\in\CC\setminus\Znn$.
It follows that $\mu=k_1\omega_1+\cdots+ k_{\ell-2}\omega_{\ell-2}+
(k_{\ell-1}+k_\ell+1)\omega_{\ell-1}+(-k_\ell-2)\omega_\ell$.
Since $\mu\in \fwlat^{0,+}\setminus\fwlat^+$, we have $k_\ell=-1$,
which forces $\mu=\lam$.
Thus, we have proved the assertion.
\end{proof}

Let $\lam$ be an element of $\fwlat^{0,+}\setminus\fwlat^+$.
We set $\sing(\lam):=\singtwo(\Phi^{\semi}(\fmdl_\lam))$.
Since any composition factor of $\Phi(\fmdl_\lam)$ which is a simple cuspidal module is a direct summand of $\Phi(\fmdl_\lam)$, together with \Cref{sec:coherentcuspidal}, we have the following corollary.

\begin{corollary}\label{thm:cuspidal}
The collection of
$$
\Phi_{\mu+\frlat}(\fmdl_\lam)\qquad (\lam\in \fwlat^{0,+}\setminus\fwlat^+,
\mu+\frlat\in(\fh^*/\frlat)\setminus \sing(\lam)),
$$
is the complete set of simple cuspidal modules over $\fg$.
Moreover, if $\lam\in\fwlat^+$, then $\Phi(\fmdl_\lam)=0$.
\end{corollary}

\section{Relaxed \hw  modules with cuspidal top spaces over $\uaffvoa{k}{\fg}$}\label{sec:relaxed}

\subsection{Relaxed \hw  module}\label{sec:relaxed1}

Let $\mdm$ be a weight $\uaffvoa{k}{\fg}$-module.
A {\em relaxed \hw  vector} is a weight vector $v\in\mdm$ such that
$(\fg[t]t).v=0$.
If $\mdm$ is generated by a relaxed \hw  vector, $\mdm$ is called 
a {\em relaxed \hw  module}.

Let $\mdm$ be a relaxed \hw  module with a relaxed \hw  vector $v\in\mdm$.
Then $\mdm$ is a positive energy module and the top space $\mdm_{top}$ of $\mdm$ is $\envalg{\fg}.v$.
The {\em almost simple quotient} of $\mdm$ is the quotient module
$\mdm/\mdn$ of $\mdm$ by the sum $\mdn$ of all submodules of $\mdm$ which intersects
$\mdm_{top}$ trivially.
The $\mdm$ is called {\em almost simple} if it is an almost simple quotient.
An almost simple module is simple if and only if the top space is simple.

Let $\fmdm$ be a weight $\fg\oplus \CC D\oplus \CC K$-module of level $k$.
We define the induced module
$$
\mdv_{\fmdm}:=\envalg{\ag}\otimes_{\fg[t]\oplus\CC K\oplus \CC D}\fmdm.
$$
If $\fmdm$ is generated by a single element of $\fmdm$, then $\mdv_{\fmdm}$
is a relaxed \hw  module and called a {\em relaxed Verma module}.
The almost simple quotient of $\mdv_{\fmdm}$ is denoted by
$\mdl_{\fmdm}$.


\begin{definition}
A relaxed \hw  module $\mdm$ of $\uaffvoa{k}{\fg}$ has a {\em cuspidal top space} if $\zhu{\mdm}$ is a cuspidal module.
\end{definition}

When $\fg=A_1$, relaxed \hw  modules with cuspidal top spaces 
are fundamental in \cite{CreMod13}.

\subsection{Twisted localisation for $\ag$}\label{sec:tw}
Recall that $S$ is the multiplicative subset of $\ufg$ generated by $f_{\beta_1},\ldots,f_{\beta_\ell}$.
The set $S$ is embedded in $\ug$.
As in \Cref{sec:mathieu}, let us consider localisation $S^{-1}\envalg{\ag}$ of
the universal enveloping algebra of $\ag$.

Let $\Lam$ be an element of 
$\wlat^{0,+}$.
We set $\mdm=\mdv^0_\Lam$ or $\mdl_\Lam$.
Setting $\lam:=\fLam$, we see that the top space of $\mdm$ is the $\fg$-module 
$\fmdm:=\fmdv^0_{\lam}, \fmdl_{\lam}$, respectively.

We define localisation of $S^{-1}\mdm$ as before.
The module $S^{-1}\mdm$ is non-zero if $\mdm=\mdv^0_\Lam$ or $\lam\not\in \fwlat^+$.
The twist of $f^x.S^{-1}\fmdm$ for $x\in\CC^\ell$ is also defined as in \Cref{sec:mathieu}.
The twisted localisation $\Phi(\mdm)$ of $\mdm$ is similarly defined to be
$$
\Phi(\mdm)=\bigoplus_{x=(x_1,\ldots,x_\ell)\in\CC^\ell, 0\leq Re(x_1),\ldots,Re(x_\ell)<1}
f^x.S^{-1}\mdm.
$$
The top space of $\Phi(\mdm)$ is isomorphic to $\Phi(\fmdm)$.
Let $\mu+\frlat$ be an element of $\fh^*/ \frlat$.
We set $\Phi_{\mu+\frlat}(\mdm):=\bigoplus_{\nu\in\mu+\frlat+k\Lam_0+\CC \delta}\Phi(\mdm)(\nu)$.

The following proposition is proved similarly as \cite[Lemma 13.2]{MatCla00}.

\begin{proposition}\label{sec:generates}
$\Phi(\mdm)$ is generated by the top space $\Phi(\fmdm)$ as a $\ag$-module.
\end{proposition}

Now let us consider $\mdm=\mdl_\Lam$.
We see that $\Phi(\mdl_\Lam)$ is isomorphic to the almost simple relaxed \hw  module $\mdl_{\Phi(\fmdl_\lam)}$:
\begin{equation}\label{eqn:commuteloc}
\Phi(\mdl_\Lam)\cong \mdl_{\Phi(\fmdl_\lam)}.
\end{equation}

For the completeness, we give proofs of \Cref{sec:generates} and \eqref{eqn:commuteloc}
in \Cref{sec:appendix1}.

Let $\mu+\frlat$ be an element of $\fh^*/ \frlat$.
Then $\Phi_{\mu+\frlat}(\mdl_\Lam)$ is isomorphic to the almost simple
relaxed \hw  module $\mdl_{\Phi_{\mu+\frlat}(\fmdl_\lam)}$ with a cuspidal top space.
Therefore, (a) in \cref{sec:thmsummary1} 
follows from \Cref{thm:cuspidal}.


Let $\lam$ be an element of $\fwlat^{0,+}\setminus\fwlat^+$.
The following proposition is clear and the proof is given in \Cref{sec:appendix2}.

\begin{proposition}\label{sec:duflo}
The module $\fmdl_\lam$ is a $\zhu{\affvoa{k}{\fg}}$-module if and only if 
so is $\Phi(\fmdl_\lam)$.
\end{proposition}

Thus, we have the following theorem by \cref{sec:zhu}.

\begin{theorem}\label{sec:simplecoherent}
Let $\Lam$ be an element of $\wlat^{0,+}_k$ such that $\fLam\not\in\fwlat^+$.
The $\uaffvoa{k}{\fg}$-module $\Phi(\mdl_\Lam)$ is a $\affvoa{k}{\fg}$-module if and only if
so is $\mdl_\Lam$.
\end{theorem}

\begin{corollary}\label{sec:simplecuspidal}
The list of all $\Phi_{\mu+\frlat}(\mdl_\Lam)$ with data
\begin{enumerate}
\item $\Lam\in \wlat^{0,+}_k$ such that $\fLam\not\in \fwlat^+$ and $\mdl_\Lam$ is a $\affvoa{k}{\fg}$-module, and
\item $\mu+\frlat\in(\fh^*/\frlat)\setminus\sing(\fLam)$
\end{enumerate}
is the complete list of simple relaxed \hw  modules with cuspidal top spaces over $\affvoa{k}{\fg}$.
\end{corollary}

Thus, from now on, we will study $\Phi(\mdl_\Lam)$ with $\Lam\in \wlat^{0,+}_k$.

\subsection{Characters of relaxed \hw  modules}\label{sec:char1}

Let $\mdm$ be a weight module over $\uaffvoa{k}{\fg}$.
The {\em string function} of $\mdm$ at $\lam\in\fh^*$ is the $q$-series
$$
s_\lam[\mdm](q)=\sum_{n\in\CC}\dim \mdm(k\Lam_0+\lam-n\delta)q^{n-c/24},
$$
where $c$ is the central charge of $\uaffvoa{k}{\fg}$ \eqref{eqn:caff}.
The {\em (full) character} of $\mdm$ is
$$
\ch[\mdm](y,q,z)=y^k \sum_{\lam\in\fh^*}s_\lam[\mdm](q)z^\lam,
$$
For a weight $\fg$-module $\fmdm$, we also define the character
$$
\ch[\fmdm](z)=\sum_{\lam\in\fh^*}\dim\fmdm(\lam)z^\lam.
$$

Let $\Lam$ be an element of $\wlat^{0,+}_k$ and set $\lam=\fLam$.
Recall the irreducible finite-dimensional module $\fmdh_\lam$ over $\fg_0$
in \cref{sec:graded} and Euler's $\varphi$ series $\varphi(q)=\prod_{n=1}^\infty (1-q^n)$.
The characters of the parabolic Verma modules $\fmdv^0_\lam$ and $\mdv^0_\Lam$ are given by
\begin{align*}
&\ch[\fmdv^0_\lam](z)=\frac{\ch[\fmdh_\lam](z)}
{(1-z^{-\beta_1})\cdots(1-z^{-\beta_\ell})},\\
&\ch[\mdv^0_\Lam](y,q,z)=\frac{y^kq^{h_\Lam-c/24}\ch[\fmdv^0_\lam](z)}
{ \varphi(q)^\ell\prod_{i\geq1,\alpha\in\fDelta}(1-q^iz^\alpha)},
\end{align*}
where $h_\Lam=\Lam(-D)$.
Therefore, setting $\beta_0=\beta_1+\cdots+\beta_\ell$, we see that the following limiting string function exists and has the form
\begin{equation}\label{eqn:limverma}
\lim_{n\rightarrow -\infty}s_{\lam+n\beta_0}[\mdv^0_\Lam](q)=q^{h_\Lam-c/24}\frac{\deg\fmdv^0_\lam}{\varphi(q)^{\dim\fg}}.
\end{equation}
(Recall \eqref{eqn:deg}.)

Let $\mdm$ be a module in category $\cO^0_k$ whose weight support as a $\fg$-module
lies in a single coset in $\fh^*/\frlat$ and let $\lam$ be a representative of the weight support
of $\mdm$: 
$\supp{\mdm}\subset \lam+\frlat$.
By the existence of the limiting string function of the parabolic Verma modules, we see that the following limiting string function exists and does not depend on the choice of the representative $\lam$ of $\supp{\mdm}$:
$$
s_{-\infty}[\mdm](q):=\lim_{n\rightarrow -\infty}s_{\lam+n\beta_0}[\mdm](q).
$$

The following lemma is clear.

\begin{lemma}\label{sec:limiting}
$$
\ch[\Phi(\mdm)](y,q,z)=y^k s_{-\infty}[\mdm](q)\sum_{\mu\in\frlat}z^\mu.
$$
\end{lemma}

Let $\Lam$ be an element of 
 $\wlat^{0,+}_k$.
By \eqref{eqn:parabolicKL}, we have
$$
s_{-\infty}[\mdl_\Lam](q)=\sum_{\Omega\in\wlat^{0,+}_k}c_{\Lam,\Omega}
s_{-\infty}[\mdv^0_\Omega](q),
$$
where the sum in the left-hand side is well-defined, which follows from the linkage principle described in \cref{sec:parakl}.
Thus, we have the following theorem.

\begin{theorem}\label{sec:coherentKL}
The character of $\Phi(\mdl_\Lam)$ satisfies the formula
$$
\ch[\Phi(\mdl_\Lam)](y,q,z)=\sum_{\Omega\in\wlat^{0,+}_k}c_{\Lam,\Omega}
\ch[\Phi(\mdv^0_\Omega)](y,q,z),
$$
\end{theorem}

We finally give a character formula of $\Phi(\mdv^0_\Lam)$ which follows from \cref{sec:limiting}
and \eqref{eqn:limverma}.

\begin{proposition}\label{sec:parabolicverma}
$$
\ch[\Phi(\mdv^0_\Lam)](y,q,z)=y^kq^{h_\Lam-c/24}\frac{\deg \fmdv^0_{\lam}}{\varphi(q)^{\dim\fg}}
\sum_{\mu\in\fh^*}z^\mu,
$$
where $\lam:=\fLam$ and $h_\Lam:=\Lam(-D)$ is the minimal conformal weight of $\mdv^0_\Lam$.
\end{proposition}

In particular, $\Phi_{\mu+\frlat}(\mdv^0_\Lam)$ is a relaxed Verma module for all $\mu\in\fh^*$.

\section{Minimal W-algebras of type A}\label{sec:minimal}

In this Section, we recall representation theory of minimal W-algebras and show a parabolic
Kazhdan-Lusztig type formula.

\subsection{Affine W-algebras}

Let us consider the minimal nilpotent element $f=f_{\alpha_\ell}$ of $\fg=A_\ell$
and a non-critical level $k\in\CC$.
Then the affine W-algebra $\wuniva$ associated to the grading \eqref{eqn:goodgrading}
is defined as follows \cite{KRW,Aii}.

Let $\Cl$ be the Clifford algebra generated by the odd vectors
$$
\psi_\beta(n)\qquad (\beta\in\pm\fDelta_1,n\in\ZZ)
$$
subject to the relation
$$
[\psi_\alpha(m),\psi_\beta(n)]=\delta_{\alpha+\beta,0}\delta_{m+n,0}
\qquad (\alpha,\beta\in\pm\fDelta_1,m,n\in\ZZ).
$$
Here, $[,]$ denotes the supercommutator.

Let $\cF$ be the irreducible representation over $\Cl$ generated
by the vector $\bm{1}$ with
$$
\psi_\beta(n)\bm{1}=0\quad (\beta\in\fDelta_1, \beta+n\delta\mbox{ is 
a positive root of $\ag$}).
$$
The module $\cF$ is a vertex (super)algebra with the vacuum element $\bm{1}$
and generating fields
$$
\psi_\beta(z)=\sum_{n\in\ZZ}\psi_\beta(n)z^{-n-1},\quad
\psi_{-\beta}(z)=\sum_{n\in\ZZ}\psi_{-\beta}(n)z^{-n}\quad
(\beta\in\fDelta_1).
$$

The vertex algebra $\cF$ carries the {\em charge grading}
$\cF=\bigoplus_{i\in\ZZ}\cF^i$, where $\bm{1}\in\cF^0$ and
the operators $\psi_\beta(n)$ and $\psi_{-\beta}(n)$ with $\beta\in\fDelta_1$
are homogeneous of charges 1 and $-1$, respectively.

The tensor product vertex algebra
$$
\comp=\uaffvoa{k}{\fg}\otimes \cF
$$
carries the induced charge grading:
$$
\comp=\bigoplus_{i\in \ZZ}\comp^i,\qquad \comp^i:=\uaffvoa{k}{\fg}\otimes \cF^i.
$$

Let $\bar\chi$ be the element of $\fg^*$ defined by $\bar\chi(v)=(f|v)$ ($v\in\fg$)
and $Q(z)$ the odd field on $\comp$ defined as follows:
$$
Q(z)=Q^{\mathrm{st}}(z)+\chi(z),
$$
where
\begin{align*}
&Q^{\mathrm{st}}(z)=\sum_{\beta\in\fDelta_1}J_\beta(z)\psi_{-\beta}(z)
-\frac{1}{2}\sum_{\alpha,\beta,\gamma\in\fDelta_1}c_{\alpha,\beta}^\gamma
\psi_{-\alpha}(z)\psi_{-\beta}(z)\psi_\gamma(z),\\
&\chi(z)=\sum_{\beta\in\fDelta_1}\bar\chi(J_\beta)\psi_{-\beta}(z).
\end{align*}
We see that $Q_{(0)}^2=0$ and $Q_{(0)}$ is homogeneous of charge 1 with respect to the charge grading on $\comp$.
Therefore, we have a cochain complex $(\comp^\bullet,Q_{(0)})$ on $\comp$.

We now define the universal affine W-algebra $\wuniva$
to be the zeroth cohomology
$$
\wuniva=H^0(\comp^\bullet,Q_{(0)})
$$
of  $(\comp^\bullet,Q_{(0)})$ with the induced vertex algebra structure from $\comp$.
We have the hamiltonian $H_W=-D+H_\cF+\partial x$ on $\wuniva$, where $H_\cF$ is the standard
hamiltonian on $\cF$.
The W-algebra is a VOA of central charge \cite[Theorem 2.2]{KRW}
\begin{align}\label{eqn:cw}
c_W&=\dim(\fg_0)-\frac{1}{2}\dim(\fg_{1/2})-\frac{12}{k+h^\vee}|\frho-(k+h^\vee)\omega_\ell |^2
\nonumber\\
&=\frac{\ell((\ell^2-12)k^2+\ell k+10(\ell+1)^2)}{(k+\ell+1)(\ell+1)}.
\end{align}

\subsection{Finite W-algebras and representation theory}

The {\em finite W-algebra} $\wfina$ is an associative algebra
which may be constructed by quantum hamiltonian reduction.
The Zhu algebra of $\wuniva$ is isomorphic to $\wfina$ \cite{DSK}.

In this Subsection, we recall representation theory of $\wuniva$.

Let $\wfina\mbox{-mod}$ denote the category of finite-dimensional representations over $\wfina$.
Recall that we have the exact functor
$$
\hlie_0:\fO^0\longrightarrow \wfina\mbox{-mod}.
$$

By \cite{BK}, we see that for $\lam\in \fwlat^{0,+}$, 
$\hlie_0(\fmdl_\lam)\neq 0$ if and only if $\langle \lam+\rho,\alpha_\ell^\vee\rangle\not\in\ZZ_{>0}$.
By \eqref{eqn:p0+}, we see that it is equivalent to $\lam\not\in \fwlat^+$.
Along with other results of \cite{BK}, we have the following theorem.

\begin{theorem}\label{sec:finitew}
The modules $\hlie_0(\fmdl_\lam)$ with $\lam\in\fwlat^{0,+}\setminus\fwlat^+$ are
the complete list of finite-dimensional simple modules over $\wfina$.
If $\lam\in\fwlat^+$, then $\hlie_0(\fmdl_\lam)$ is zero.
\end{theorem}

\begin{proof}
The assertion follows from \cite{BK} and a similar argument as in the proof of \Cref{thm:coherent}.
\end{proof}


Finally, by \cite[Lemma 8.16]{BK} and \eqref{eqn:deg}, we have the following dimension formula of reduction of parabolic Verma modules.

\begin{lemma}\label{sec:bkparabolic}
If $\lam\in\fwlat^{0,+}$, then
$\dim \hlie_0(\fmdv^0_\lam)=\deg \fmdv^0_\lam$.
\end{lemma}

\subsection{Representations of affine W-algebras}\label{sec:w}

In this Subsection, we recall representation theory of ordinary modules over $\wuniva$ 
from \cite{Aii}.
Let $\wunivamod$  denote the category of ordinary modules over $\wuniva$.

By the result of \cite{Aii}, we have the exact functor (called the ``$-$'' reduction functor)
$$
\hbrst_0:\cO^0_k\longrightarrow \wunivamod,
$$
which sends a module $\mdm$ in category $\cO^0_k$ to the zeroth cohomology
of $\mdm\otimes \cF$ with a certain differential.
See also \cite{AE} for the definition of the functor.
Let $\mdm$ be a highest weight module in $\cO^0_k$ of highest weight $\Lam\in\wlat^{0,+}_k$.
The hamiltonian on $\hbrst_0(\mdm)$ is defined to be
$$
H_W=-D+H_\cF-\frac{k+h^\vee}{2}|\omega_\ell|^2+(\omega_\ell|\frho)
$$
(cf.~\cite{Aii,AE}).
In particular, we have
\begin{equation}\label{eqn:minimalcw}
h_\Lam^W=h_\Lam-\frac{k+h^\vee}{2}|\omega_\ell|^2+(\omega_\ell|\frho),
\end{equation}
where
$h_\Lam^W$ and $h_\Lam=\Lam(-D)$ are minimal conformal weights of $\hbrst_0(\mdm)$ and $\mdm$.

\begin{remark}
The equality \eqref{eqn:minimalcw}
holds if we replace $h_\Lam^W$ and $h_\Lam$ with the minimal conformal weights of $\hbrst_0(\mdm)$ and $\mdm$
with respect to the zero-modes of the conformal vectors of $\wuniva$
and $\uaffvoa{k}{\fg}$. See, e.g., \cite[eq.~(7.4)]{AE}.
\end{remark}

Note that if $\mdm$ is a simple $\uaffvoa{k}{\fg}$-module, then
$\zhu{\hbrst_0(\mdm)}\cong\hlie_0(\zhu{\mdm})$.

Let $\Lam$ be an element of $\wlat^{0,+}_k$.
The modules $\hbrst_0(\mdv^0_\Lam)$ and $\hbrst_0(\mdl_\Lam)$
are generated by their top spaces.

By using \Cref{sec:finitew},
we see that the collection of $\hbrst_0(\mdl_\Lam)$ with 
 $\Lam\in\wlat^{0,+}_k$ such that $\fLam\not\in\fwlat^+$ is the complete list
 of simple ordinary modules over $\wuniva$.


\subsection{Characters of ordinary modules}\label{sec:charord}

Let $\Lam$ be an element of $\wlat^{0,+}_k$.
By applying the exact functor $\hbrst_0$ to \eqref{eqn:parabolicKL}, we see that
\begin{equation}\label{eqn:charw1}
\ch[\hbrst_0(\mdl_\Lam)](q)=\sum_{\Omega\in\wlat^{0,+}_k}
c_{\Lam,\Omega}\ch[\hbrst_0(\mdv^0_\Omega)](q).
\end{equation}
The sum in the RHS is well-defined because of the linkage principle described in \cref{sec:parakl}.

Set $\lam=\fLam$ and by applying the BGG resolution to $\fmdh_{\lam}$ over $\fg_0$,
we have
\begin{equation}\label{eqn:bgg}
[\fmdh_{\lam}]=\sum_{w\in \bar\aaw_0}b_w^\lam[\fmdv_{0,w\circ\lam}],
\qquad b_{\id}^\lam=1, \qquad b_{w}^\lam\in\ZZ,
\end{equation}
where $\bar\aaw_0$ is the Weyl group of $A_{\ell-1}=[\ag_0,\ag_0]$ and $\fmdv_{0,\omega}$ is the Verma module over $\fg_0$ of highest-weight $\omega\in \fh^*$.
By multiplying both sides by $\envalg{\ag}\otimes_{U(\fg_{\geq0}\oplus \fg[t]t\oplus \CC K\oplus \CC D)}$, 
we have
\begin{equation}\label{eqn:parabgg}
[\mdv^0_\Lam]=\sum_{w\in \bar\aaw_0}b_w^\lam[\mdv_{w\circ\Lam}].
\end{equation}

Using the Euler-Poincare principle and \eqref{eqn:parabgg} to $\hbrst_0(\mdv^0_\Lam)$
as in \cite{KRW,Aii} and others, 
we see that
\begin{equation}\label{eqn:charw2}
\ch[\hbrst_0(\mdv^0_\Lam)](q)=q^{h_\Lam^W-c_W/24}\frac{\dim\hlie_0(\fmdv^0_\lam)}{\varphi(q)^{\ell^2}}\qquad (\lam:=\fLam).
\end{equation}
Note that by \cite{Aii},  $\hbrst_0(\mdv^0_\Lam)$ is generated by its top space, 
which is isomorphic to $\hlie_0(\fmdv^0_\lam)$.
Since $\wuniva$ has $\ell^2$ $(=\dim\fg^{f_{\alpha_\ell}})$ free strong generators,
it follows from \eqref{eqn:charw2} that $\hbrst_0(\mdv^0_\Lam)$ is freely generated by
its top space.

Combining \eqref{eqn:charw1}--\eqref{eqn:charw2} and \Cref{sec:bkparabolic}, we have the following character formula.

\begin{theorem}\label{sec:charw}
$$
\ch[\hbrst_0(\mdl_\Lam)](q)=q^{h_\Lam^W-c_W/24}
\sum_{\Omega\in\wlat^{0,+}_k}
\frac{c_{\Lam,\Omega}\deg \fmdv^0_\omega}{\varphi(q)^{\ell^2}},
$$
where $\omega:=\bar \Omega$.
\end{theorem}

Let $\eta(q)=q^{1/24}\varphi(q)$ denote the Dedekind eta series.
Since $\dim\fg=\ell(\ell+2)$, 
by combining \eqref{eqn:caff}, \eqref{eqn:cw}--\eqref{eqn:minimalcw}, \Cref{sec:coherentKL,sec:parabolicverma,sec:charw},
we have (c) in \cref{sec:thmsummary1}, which completes the proof of \cref{sec:thmsummary1}.

\section{Modular invariance of string functions}\label{sec:main}

\subsection{Proof of \cref{sec:modularinv}}

When $\ell\geq 2$ and the level $k$ is admissible, the simple W-algebra $\wsimpa$ is rational if 
and only if $\fg=A_2$ and the denominator of $k$ is 2, which is studied in this section.

Let $\fg$ be of type $A_2$ and $k$ an admissible level with denominator 2.
Then $k$ has the form $k+3=p/2$ with an odd $p\geq 3$.
Let $\adm_k\subset \ah^*_k$ denote the set of admissible weights of level $k$ in the sense of Kac and Wakimoto \cite{KW88}.
By the result of \cite{A16}, the simple modules over $\affvoa{k}{\fg}$ in category $\cO_k$ 
are precisely $\mdl_\Lam$ with $\Lam\in\adm_k$.
By \cref{sec:simplecoherent}, we see the following theorem.

\begin{theorem}\label{sec:thmsummary}
The list of all $\Phi_{\mu+\frlat}(\mdl_\Lam)$ with 
$\Lam\in \wlat^{0,+}_k\cap\adm_k$ such that $\fLam\not\in \fwlat^+$ and 
$\mu+\frlat\in(\fh^*/\frlat)\setminus\sing(\fLam)$
is the complete list of simple relaxed \hw  modules with cuspidal top spaces over $\affvoa{k}{\fg}$.
\end{theorem}

%

The simple W-algebra $\wsimpa$ is rational and isomorphic to the simple Bershadsky-Polyakov vertex algebra of level $k$ \cite{A13}.
The hamiltonians on $\wsimpa$-modules
are replaced by the zero mode of the conformal vector of $\wsimpa$.
Since $\wsimpa$ is rational, the $q$-characters of ordinary modules of $\wsimpa$ 
have modular invariance property. 
By \cite[Theorem 8.6]{AE}, we see that the complete list of simple ordinary module over $\wsimpa$
is given by
$
\hbrst_0(\mdl_\Lam)
$
with $\Lam\in \wlat^{0,+}_k\cap\adm_k$ such that $\fLam\not\in \fwlat^+$.
Thus, \cref{sec:thmsummary} with the character formula in \cref{sec:thmsummary1} implies the modular invariance \cref{sec:modularinv}.

\subsection{Creutzig-Ridout type modular invariance and future perspective}

Let $\fg$ be of type $A_1$ and $k$ admissible.
In \cite{CreMod13}, a modular invariance property of the full characters of certain (generically simple) 
relaxed highest weight modules with cuspidal top spaces (and their spectral flows) 
over $\affvoa{k}{A_1}$,
called the {\em standard modules}, are proposed.
For this purpose, modular invariance of string functions similar to \cref{sec:modularinv} is essentially used.

The modular invariance of full characters is used to obtain a conjectural {\em standard Verlinde formula} of standard modules over $\affvoa{k}{A_1}$ in \cite{CreMod13}.
Crucially, non-negative integers are obtained from the standard Verlinde formula
and the numbers are expected to be (Grothendieck) fusion coefficients among the standard modules over $\affvoa{k}{A_1}$.
When $k=-4/3$ and $k=-1/2$, the fusion coefficients are explicitly computed by the so-called
Nahm-Gaberdiel-Kausch algorithm and they are consistent with the ones predicted from
the standard Verlinde formula.

Note that the simple highest weight modules over $\affvoa{k}{A_1}$ have (formal)
resolutions with respect to standard modules over $\affvoa{k}{A_1}$ and 
they indeed converge if $k<0$.
The (Grothendieck) fusion rules among simple highest weight modules derived from
the (formal) resolutions and conjectural fusion rules among standard modules
are also non-negative \cite{CreMod13}.

We expect that \cref{sec:modularinv} should imply a modular invariance of the full characters of 
relaxed highest weight modules with cuspidal top spaces over $\affvoa{k}{A_2}$ of admissible
level $k$ with denominator 2.
We hope to come back to this point in future works including \cite{KRW}.

We also hope to generalise the method proposed in this article to more general
gradings of simple Lie algebras and general relaxed highest weight modules.
Since there are many rational W-algebras associated to general gradings (see, e.g., \cite{AE}),
it is expected that we have many modular invariant string functions of relaxed highest weight modules.


%

\appendix
\section{Proofs}\label{sec:appendix}

\subsection{Almost simplicity of $\Phi(\mdl_\Lam)$}
\label{sec:appendix1}

In this Subsection, we prove  \Cref{sec:generates} and \eqref{eqn:commuteloc}.
Set $f_i=f_{\beta_i}$, $1\leq i\leq \ell$.

\begin{proof}[Proof of \Cref{sec:generates}]
Let $v$ be an element of $\Phi(\mdm)$.
Then $v$ is a finite sum of the vectors of the form
$f^x.(aw)$ with $x\in\CC^\ell$, $a\in \ug$ and $w\in \fmdm$.
It suffices to show $f^x.(aw)\in \ug \Phi(\fmdm)$.
Indeed,
\begin{align*}
f^x.aw&=(f^xa f^{-x})f^x.w\\
&=\sum_{j_1,\ldots,j_\ell=0}^\infty
\begin{pmatrix}x_1\\j_1\end{pmatrix}\cdots \begin{pmatrix}x_\ell\\j_\ell\end{pmatrix}ad(f_1)^{j_1}\cdots ad(f_\ell)^{j_\ell}(a)f_\ell^{-j_\ell}\cdots f_1^{-j_1}f^x.w
\end{align*}
and since $f_\ell^{-j_\ell}\cdots f_1^{-j_1}f^x.w\in \Phi(\fmdm)$, it follows that
$f^x.aw\in\ug\Phi(\fmdm)$.
%
%
\end{proof}


\begin{proof}[Proof of \eqref{eqn:commuteloc}]
By \cref{sec:generates}, we have surjective homomorphisms
$$
\mdv_{\Phi(\fmdl_\lam)}\longrightarrow \Phi(\mdl_\Lam)\longrightarrow \mdl_{\Phi(\fmdl_\lam)}.
$$
Let $\Mod{J}$ be the sum of all $\ag$-submodules of $\Phi(\mdl_\Lam)$
which have zero intersection with the top space.
Since $\Phi(\Mod{J})$ has zero intersection with the top space again, 
we have $\Mod{J}=\Phi(\Mod{J})$.
Let $v$ an element of $\Phi(\Mod{J})$.
Then there is $x\in\CC^\ell$ such that $f^x.v\in\mdl_\Lam$.
Since $\mdl_\Lam$ is simple, it follows that $f^x.v=0$, which implies $v=0$.
Thus, we have proved \eqref{eqn:commuteloc}.
\end{proof}

\subsection{Annihilators}\label{sec:appendix2}
In this Subsection, we show \cref{sec:duflo}.
Let $\ann{\fmdm_\lam}{\ufg}$ denote the space of annihilators of 
a $\ufg$-module $\fmdm$.

\begin{proposition}\label{sec:annloc}
$\ann{\fmdl_\lam}{\ufg}=\ann{S^{-1}\fmdl_\lam}{\ufg}$
\end{proposition}

\begin{proof}
Since $\fmdl_\lam\subset S^{-1}\fmdl_\lam$, we have $\ann{\fmdl_\lam}{\ufg}\supset\ann{S^{-1}\fmdl_\lam}{\ufg}$.
Now, let $a$ be an element of $\ann{\fmdl_\lam}{\ufg}$ and $v$ an element of $S^{-1}\fmdl_\lam$.
It suffices to show $av=0$. 
Setting $g=f_1\cdots f_\ell$,
 we have $r\geq0$ such that $g^rv\in \fmdl_\lam$.
By ad-nilpotency of $\ufg$, we have $t\geq0$ which satisfies $ad(g^r)^t(a)=0$.
Then we have
\begin{align*}
g^{rt} av&=g^{r(t-1)}ad(g^r)(a)v+g^{r(t-1)}ag^rv=g^{r(t-1)}ad(g^r)(a)v\\
&=\cdots=ad(g^r)^t(a)=0.
\end{align*}
Since $g$ acts injectively on $S^{-1}\fmdl_\lam$, we have $av=0$, as desired.
\end{proof}

\begin{lemma}\label{sec:anntw}
Let $\fmdm$ be a $S^{-1}\ufg$-module and $x$ an element of $\CC^\ell$. 
Then $\ann{\fmdm}{\ufg}=\ann{f^x.\fmdm}{\ufg}$.
\end{lemma}

\begin{proof}
Since $f^{-x}.f^x.\fmdm\cong \fmdm$, it suffices to show $\ann{\fmdm}{\ufg}\subset \ann{f^x.\fmdm}{\ufg}$.
Let $a$ be an element of $\ann{\fmdm}{\ufg}$ and $s^x.v$ an element of $f^x.\fmdm$ with $v\in \fmdm$.
Then 
\begin{align*}
af^x.v&=f^x(f^{-x}af^x)v\\
&=f^x.\sum_{j_1,\ldots,j_\ell=0}^\infty
\begin{pmatrix}x_1\\j_1\end{pmatrix}\cdots \begin{pmatrix}x_\ell \\ j_\ell\end{pmatrix}ad(f_1)^{j_1}\cdots ad(f_\ell)^{j_\ell}(a)f_\ell^{-j_\ell}\cdots s_1^{-j_1}v\\
&=0.
\end{align*}
Thus, we have the assertion.
\end{proof}

Combining \cref{sec:annloc,sec:anntw}, we have proved \cref{sec:duflo}.

\end{document}